\documentclass[11pt]{amsart}

\usepackage[T1]{fontenc}
\usepackage{lmodern}
\usepackage{microtype}
\usepackage{amsmath,amssymb,amsthm,mathtools}
\usepackage{enumitem}
\usepackage{hyperref}

\newtheorem{theorem}{Theorem}[section]
\newtheorem{proposition}[theorem]{Proposition}
\newtheorem{lemma}[theorem]{Lemma}
\newtheorem{conjecture}[theorem]{Conjecture}
\newtheorem{problem}[theorem]{Problem}
\theoremstyle{remark}

\newcommand{\Q}{\mathbf Q}
\newcommand{\Z}{\mathbf Z}
\newcommand{\Pj}{\mathbf P}
\newcommand{\Gal}{\operatorname{Gal}}

\newcommand{\Sym}{\operatorname{Sym}}

\title{Van der Waerden's Problem for Cubic Surfaces}
\author{Evan O'Dorney}
\date{October 7, 2026. Written by ChatGPT 6.}

\begin{document}

\begin{abstract}
For a smooth cubic surface $S/\Q$, the absolute Galois group acts on the $27$ lines of $S_{\overline{\Q}}$, giving a subgroup
\[
G_S\subseteq W(E_6).
\]
We study the distribution of $G_S$ when the twenty coefficients of a defining cubic form are chosen independently from a large interval $[-H,H]$.

We prove that, for every $\varepsilon>0$, the number of smooth integral cubic forms of height at most $H$ for which
\[
G_S\neq W(E_6)
\]
is
\[
O_\varepsilon\!\left(H^{19+1/21+\varepsilon}\right).
\]
The principal obstruction is the unique subgroup $W(E_6)^+$ of index $2$. By a theorem of Saito, membership in this subgroup is equivalent to the signed discriminant of the cubic surface being a square. A recent square-value estimate of Bonolis, Pierce, and Woo then gives the exponent $19+1/21$. For each of the other five conjugacy classes of maximal subgroups of $W(E_6)$, recent results on rational points in projective thin sets give the sharper bound $O(H^{19})$.

We also prove the lower bound
\[
N_{\mathrm{exc}}(H)\gg H^{16}
\]
by counting cubic surfaces containing a fixed rational line, and exhibit an $H^{14}$-sized exceptional family with no rational line, arising from rational Eckardt points. These examples lead us to conjecture that
\[
N_{\mathrm{exc}}(H)=H^{16+o(1)}.
\]

Finally, we investigate whether Bhargava's proof of van der Waerden's conjecture for $S_n$ can be adapted to the index-$2$ $E_6$-problem. A double-discriminant argument gives the desired $H^{19+\varepsilon}$ bound whenever the product of the strongly square-dividing primes is sufficiently large. The remaining difficulty comes from weak square divisibility.
\end{abstract}

\maketitle

\section{Introduction}

For a smooth cubic surface with integer coefficients, the Galois group of the $27$ lines is a subgroup of the Weyl group $W(E_6)$. In the spirit of van der Waerden, we ask:

\begin{quote}
If the $20$ coefficients are chosen randomly from a large interval $[-H,H]$, what is the probability that the resulting Galois group is strictly smaller than $W(E_6)$?
\end{quote}

The analogous problem for ordinary polynomials goes back to van der Waerden. If
\[
f(x)=x^n+a_1x^{n-1}+\cdots+a_n,
\qquad
\lvert a_i\rvert\le H,
\]
then the maximal possible Galois group is $S_n=W(A_{n-1})$. Van der Waerden conjectured in 1936 that the number of polynomials for which
\[
\Gal(f)\neq S_n
\]
is $O(H^{n-1})$. This was recently proved for all $n$ by Bhargava \cite{Bhargava2025}.

There is a parallel type-$B$ problem. A reciprocal polynomial of degree $2n$ has its roots grouped into pairs
\[
\{\alpha,\alpha^{-1}\},
\]
so its generic Galois group is the hyperoctahedral group
\[
C_2\wr S_n=W(B_n).
\]
Anderson, Bertelli, and O'Dorney proved that the proportion of reciprocal polynomials whose Galois group is smaller than $W(B_n)$ is of order
\[
\frac{\log H}{H}.
\]
Here, unlike in the type-$A$ problem, the dominant contribution comes from an index-$2$ subgroup, characterized by a square condition on
\[
(-1)^n f(1)f(-1).
\]
See \cite{ABO2026}.

Thus general and reciprocal polynomials may be viewed as particularly natural linear families for the Weyl groups of types $A$ and $B$, respectively. The family of cubic surfaces provides an equally classical, but structurally rather different, family for the exceptional Weyl group $W(E_6)$.

Let
\[
V=\Sym^3(\Z^4)
\]
be the rank-$20$ lattice of quaternary cubic forms. For
\[
F\in V,
\]
write $S_F\subseteq\Pj^3$ for the cubic surface $F=0$. The classical configuration of the $27$ lines gives
\[
G_F:=\Gal(K_F/\Q)
\hookrightarrow W(E_6),
\]
where $K_F$ is the field over which all $27$ lines are defined. The geometric monodromy group of the universal family is the full group $W(E_6)$, of order $51840$; this is the cubic-surface case of Harris's theory of Galois groups of enumerative problems \cite{Harris1979}. In particular, Hilbert irreducibility implies that a density-one set of integral cubic surfaces has full Galois group $W(E_6)$. This qualitative density-one statement was already noted by Poonen and Voloch in their study of random Diophantine equations \cite{PoonenVoloch}. Our concern here is the corresponding quantitative problem: how rapidly does the exceptional proportion tend to zero, and what geometric families account for the dominant exceptional contribution?

There is an important difference between this situation and the classical polynomial problems. A general cubic surface does give rise, by elimination, to degree-$27$ polynomials whose roots encode the $27$ lines. However, the coefficients of these polynomials are highly nonlinear functions of the twenty coefficients of $F$. We are not aware of any linear subspace in the coefficient space of degree-$27$ polynomials whose specializations cut out all $W(E_6)$-extensions of $\Q$. Thus the natural parameter space for the problem is not a space of univariate polynomials but the linear space $V$ of cubic forms itself.

For $H\ge1$, set
\[
V(H)=
\{F\in V:\lVert F\rVert_\infty\le H\},
\]
and define
\[
N_{\mathrm{exc}}(H)
=
\#\left\{
F\in V(H):
S_F\text{ smooth and }G_F\neq W(E_6)
\right\}.
\]

Our first result is the following.

\begin{theorem}\label{thm:main}
For every $\varepsilon>0$,
\[
H^{16}\ll N_{\mathrm{exc}}(H)
\ll_\varepsilon
H^{19+1/21+\varepsilon}.
\]
More precisely, let
\[
W(E_6)^+=
\ker\!\left(
\det:W(E_6)\to\{\pm1\}
\right).
\]
Then
\[
\#\{F\in V(H):G_F\subseteq W(E_6)^+\}
\ll_\varepsilon
H^{19+1/21+\varepsilon},
\]
while the number of $F$ for which $G_F$ is contained in a proper maximal subgroup other than $W(E_6)^+$ is
\[
O(H^{19}).
\]
\end{theorem}

The exponent $1/21$ comes from a recent theorem of Bonolis, Pierce, and Woo on square values of multivariate polynomials: in $n$ variables their general estimate is
\[
H^{n-1+1/(n+1)+\varepsilon}.
\]
Here $n=20$ \cite{BPW2025}.

The remaining maximal subgroups have indices
\[
27,\quad36,\quad40,\quad40,\quad45.
\]
Recent work of Buggenhout, Cluckers, Salberger, Santens, and Vermeulen resolves Serre's projective thin-set problem in the degree range relevant here and gives the full one-power saving for these covers \cite{BCSSV2026}.

Since the number of singular cubic forms in $V(H)$ is $O(H^{19})$, Theorem~\ref{thm:main} implies
\[
\Pr(G_F\neq W(E_6)\mid S_F\text{ smooth})
\ll_\varepsilon
H^{-20/21+\varepsilon}.
\]

We regard the upper bound in Theorem~\ref{thm:main} as only a first approximation to the truth.

\begin{conjecture}\label{conj:sharp}
One has
\[
N_{\mathrm{exc}}(H)=H^{16+o(1)}.
\]
\end{conjecture}

A stronger form would assert
\[
N_{\mathrm{exc}}(H)\sim cH^{16}
\]
for an explicit positive constant $c$, with the leading term arising from surfaces having a rational line.

The lower bound $H^{16}$ is elementary. If, for example,
\[
L=\{x_0=x_1=0\},
\]
then the condition $L\subseteq S_F$ imposes four independent linear conditions on the twenty coefficients of $F$. Thus there are $\asymp H^{16}$ such forms, and a positive proportion are smooth. Since $L$ is rational, $G_F$ fixes one of the $27$ lines and hence lies in the index-$27$ maximal subgroup $W(D_5)$.

The main difficulty in improving Theorem~\ref{thm:main} is instead the index-$2$ subgroup. This parallels both the alternating group $A_n$ in Bhargava's proof and the principal index-$2$ subgroup in the reciprocal-polynomial problem.

\section{Cubic surfaces and their monodromy}

Let
\[
\Pj(V_{\Q})=\Pj^{19}.
\]
There is an irreducible discriminant hypersurface
\[
\mathcal D\subseteq\Pj^{19}
\]
parametrizing singular cubic surfaces. Its defining polynomial has degree $32$. Let
\[
U=\Pj^{19}\setminus\mathcal D.
\]

Over $U$ there is a finite \'{e}tale cover
\[
\mathcal L\longrightarrow U
\]
of degree $27$, whose fiber above $S$ is the set of lines on $S$. Its Galois closure
\[
\widetilde U\longrightarrow U
\]
has group
\[
W(E_6).
\]
The incidence relations among the lines explain why the monodromy is much smaller than $S_{27}$; Harris proved that it is nevertheless as large as these relations permit \cite{Harris1979}.

The maximal-subgroup structure is particularly favorable. Besides the unique normal subgroup
\[
W(E_6)^+
\]
of index $2$, there are five conjugacy classes of maximal nonnormal subgroups, of indices
\[
27,\;36,\;40,\;40,\;45.
\]
The index-$27$ group stabilizes a line, the index-$36$ group a double-six, one index-$40$ group stabilizes a Steiner complex or $A_2^3$-root subsystem, and the index-$45$ group stabilizes a tritangent plane. See \cite{Dolgachev2012}.

For a maximal subgroup $M<W(E_6)$, consider
\[
X_M=\widetilde U/M.
\]
Then
\[
X_M\to U
\]
has degree
\[
[W(E_6):M].
\]

For $F\in U(\Q)$,
\[
G_F
\]
is contained in a conjugate of $M$ if and only if the fiber of $X_M\to U$ above $F$ has a rational point.

Thus the exceptional Galois-group problem becomes a union of six thin-set counting problems.

\section{The five nonquadratic covers}

We first dispense with the five nonnormal maximal subgroups.

\begin{proposition}\label{prop:nonquadratic}
Let $M<W(E_6)$ be a maximal subgroup other than $W(E_6)^+$. Then
\[
\#\left\{
F\in V(H):
S_F\text{ smooth},\
G_F\subseteq gMg^{-1}
\text{ for some }g
\right\}
\ll H^{19}.
\]
\end{proposition}

\begin{proof}
Normalize $\Pj^{19}$ in the function field of $X_M$. We obtain a finite dominant cover
\[
\overline X_M\longrightarrow\Pj^{19}
\]
of degree
\[
27,\ 36,\ 40,\ 40,\text{ or }45.
\]

It is ramified. Indeed, a nontrivial connected finite cover of projective space cannot be everywhere \'{e}tale in characteristic zero; geometrically, ramification occurs over the discriminant divisor.

The theorem of Buggenhout--Cluckers--Salberger--Santens--Vermeulen on projective thin sets gives
\[
O(H^{19})
\]
rational points of projective height at most $H$ in the image, since all the cover degrees are at least $5$ \cite{BCSSV2026}.

To return from primitive projective representatives to all integral coefficient vectors, write
\[
F=dF_0,\qquad F_0\text{ primitive}.
\]
Scalar multiplication does not change the cubic surface or its lines. Thus
\[
\sum_{d\le H}
O\!\left((H/d)^{19}\right)
=O(H^{19}).
\]
\end{proof}

Accordingly, a van der Waerden bound
\[
N_{\mathrm{exc}}(H)\ll H^{19}
\]
is reduced entirely to the index-$2$ cover.

\section{The index-$2$ subgroup and the signed discriminant}

The unique normal subgroup
\[
W(E_6)^+\triangleleft W(E_6)
\]
is the kernel of the determinant character of the reflection representation:
\[
1\longrightarrow W(E_6)^+
\longrightarrow W(E_6)
\overset{\det}{\longrightarrow}
\{\pm1\}
\longrightarrow1.
\]
It is simple and has order $25920$.

Let
\[
\Delta(F)
\]
denote a fixed primitive integral normalization of the signed discriminant whose square class gives this character. It is a homogeneous polynomial of degree $32$ in the twenty coefficients of $F$.

Saito's determinant formula implies that, for a smooth cubic surface,
\[
G_F\subseteq W(E_6)^+
\quad\Longleftrightarrow\quad
\Delta(F)\in\Q^{\times2}.
\]
Since $\Delta(F)$ is integral, this is equivalent, after the fixed normalization, to
\[
\Delta(F)=y^2
\]
for some integer $y$ \cite{Saito2012}.

This characterization was also studied explicitly for cubic surfaces by Elsenhans and Jahnel, particularly in pentahedral coordinates \cite{ElsenhansJahnel2012}.

Thus define
\[
N_+(H)
=
\#\left\{
F\in V(H):
\Delta(F)\neq0,\ \Delta(F)\text{ is a square}
\right\}.
\]

The index-$2$ problem is exactly the problem of bounding $N_+(H)$.

\begin{lemma}\label{lem:irreducible}
The discriminant polynomial $\Delta$ is absolutely irreducible.
\end{lemma}

\begin{proof}
Consider the incidence variety
\[
I=
\{(F,P)\in\Pj(V)\times\Pj^3:
\nabla F(P)=0\}.
\]
For fixed $P$, the condition $\nabla F(P)=0$ consists of four linear equations on $F$. Hence $I$ is irreducible. Its image in $\Pj(V)$ is precisely the discriminant hypersurface $\mathcal D$. Therefore $\mathcal D$ is geometrically irreducible.
\end{proof}

\begin{lemma}\label{lem:nondegenerate}
There is no nonzero vector
\[
v\in V_{\Q}
\]
such that
\[
D_v\Delta\equiv0.
\]
Consequently, after no transformation in
\[
\operatorname{GL}(V_{\Q})
\]
can $\Delta$ be made independent of one of the twenty variables.
\end{lemma}

\begin{proof}
Let
\[
K=
\{v\in V_{\Q}:D_v\Delta=0\}.
\]
The group $\operatorname{SL}_4$ acts irreducibly on
\[
V_{\Q}=\Sym^3((\Q^4)^*).
\]
The discriminant is invariant under $\operatorname{SL}_4$, so $K$ is an $\operatorname{SL}_4$-stable linear subspace.

Hence either
\[
K=0
\qquad\text{or}\qquad
K=V_{\Q}.
\]
The second possibility would make $\Delta$ constant. Thus $K=0$.
\end{proof}

We can now apply Bonolis--Pierce--Woo.

\begin{theorem}\label{thm:index2}
For every $\varepsilon>0$,
\[
N_+(H)
\ll_\varepsilon
H^{19+1/21+\varepsilon}.
\]
\end{theorem}

\begin{proof}
Apply the square-value theorem of Bonolis, Pierce, and Woo to
\[
Y^2-\Delta(X_1,\ldots,X_{20}).
\]
By Lemmas~\ref{lem:irreducible} and~\ref{lem:nondegenerate}, its irreducibility and nondegeneracy hypotheses are satisfied. Their estimate in $n$ variables is
\[
H^{n-1+1/(n+1)+\varepsilon}.
\]
Putting $n=20$ gives
\[
H^{19+1/21+\varepsilon}.
\]
See \cite{BPW2025}.
\end{proof}

Combining Theorem~\ref{thm:index2} with Proposition~\ref{prop:nonquadratic} proves the upper bound in Theorem~\ref{thm:main}.

\section{A Bhargava-type approach to the quadratic cover}

The exponent $19+1/21$ comes from a general-purpose square sieve. The geometry of the cubic-surface discriminant suggests that a sharper argument should be possible.

Bhargava's proof of van der Waerden's conjecture is especially relevant. In the type-$A$ problem, the alternating group is detected by the condition that the polynomial discriminant is a square. The proof does not simply apply a square sieve to the discriminant. Instead, it distinguishes several forms of discriminant divisibility and uses, among other tools, a double-discriminant argument \cite{Bhargava2025}.

A related mechanism also appears in the reciprocal-polynomial problem, where the difficult index-$2$ subgroups are again handled by adaptations of Bhargava's discriminant methods \cite{ABO2026}.

We now record the portion of that method which transfers directly to cubic surfaces.

Choose a primitive cubic form
\[
F_0\in V
\]
with
\[
\Delta(F_0)\neq0.
\]
Extend $F_0$ to an integral basis of $V$, and write
\[
F=F_{\mathbf b}+tF_0,
\qquad
\mathbf b\in\Z^{19}.
\]
Then
\[
\Delta(F_{\mathbf b}+tF_0)
\]
is a degree-$32$ polynomial in $t$ whose leading coefficient is the nonzero constant
\[
\Delta(F_0).
\]

Define the double discriminant
\[
R(\mathbf b)
=
\operatorname{Res}_t
\left(
\Delta(F_{\mathbf b}+tF_0),
\frac{\partial}{\partial t}
\Delta(F_{\mathbf b}+tF_0)
\right).
\]
Since $\Delta$ is irreducible and has positive degree in $t$,
\[
R\not\equiv0.
\]

For an integral cubic form $F$, define
\[
C_{\mathrm{str}}(F)
=
\prod_{\substack{p\ \mathrm{prime}\\
\Delta(F)\equiv0\pmod p\\
\nabla\Delta(F)\equiv0\pmod p}}
p,
\]
ignoring a fixed finite set of primes depending only on our normalization.

The primes occurring in $C_{\mathrm{str}}$ are those for which square divisibility of the discriminant is stable under every first-order perturbation of the coefficients.

\begin{proposition}\label{prop:strong}
For every $\varepsilon>0$,
\[
\#\left\{
F\in V(H):
\Delta(F)\neq0,\
C_{\mathrm{str}}(F)>H
\right\}
\ll_\varepsilon H^{19+\varepsilon}.
\]
\end{proposition}

\begin{proof}
If $p\mid C_{\mathrm{str}}(F)$, then
\[
\Delta(F_{\mathbf b}+tF_0)
\equiv
\frac{\partial}{\partial t}
\Delta(F_{\mathbf b}+tF_0)
\equiv0\pmod p.
\]
Hence
\[
p\mid R(\mathbf b).
\]
Thus
\[
C_{\mathrm{str}}(F)\mid R(\mathbf b).
\]

Suppose first that
\[
R(\mathbf b)\neq0.
\]
Since $R$ is a fixed polynomial,
\[
\lvert R(\mathbf b)\rvert\ll H^A
\]
for some fixed $A$. It therefore has
\[
H^{o(1)}
\]
divisors.

For a squarefree divisor
\[
C\mid R(\mathbf b),
\qquad
C>H,
\]
the congruence
\[
\Delta(F_{\mathbf b}+tF_0)\equiv0\pmod C
\]
has $H^{o(1)}$ residue classes modulo $C$, since the degree in $t$ is $32$ and the leading coefficient is fixed and nonzero outside a finite set of primes. Since $C>H$, each residue class contributes $O(1)$ possible $t$ in the interval under consideration.

Thus each $\mathbf b$ contributes $H^{o(1)}$ possibilities for $t$, giving
\[
H^{19+o(1)}
\]
points.

If
\[
R(\mathbf b)=0,
\]
there are $O(H^{18})$ possibilities for $\mathbf b$, and allowing all $O(H)$ values of $t$ again gives $O(H^{19})$.
\end{proof}

This is precisely the one-power saving sought in a van der Waerden theorem.

The limitation is equally important. From
\[
p^2\mid\Delta(F)
\]
one cannot conclude
\[
\nabla\Delta(F)\equiv0\pmod p.
\]
At a smooth point of the discriminant divisor, $\Delta$ itself is a local parameter, and that parameter may take a value divisible by $p^2$ without its differential vanishing modulo $p$.

Accordingly, the unresolved part of a Bhargava-style proof is the estimate
\[
\#\left\{
F\in V(H):
\Delta(F)\text{ is a nonzero square},\
C_{\mathrm{str}}(F)\le H
\right\}
\ll_\varepsilon H^{19+\varepsilon}.
\]

For ordinary polynomials, Bhargava is able to distinguish polynomial discriminant from field discriminant: weak square divisibility often reflects nonmaximality of the order rather than genuine ramification in the number field. He then uses information about number fields and small generators.

There is not presently an equally direct object for cubic surfaces. One may eliminate the lines and obtain a degree-$27$ polynomial, but its coefficients are nonlinear functions of $F$, and there is no known linear family of degree-$27$ polynomials parameterizing the $W(E_6)$-extensions arising from cubic surfaces. Thus the small-field-discriminant portion of Bhargava's argument does not transfer formally.

This seems to be the central arithmetic problem of the paper.

\section{Lower bounds and accumulating families}

\subsection{Rational lines}

Fix the line
\[
L=\{x_0=x_1=0\}.
\]
The restriction
\[
F|_L
\]
is a binary cubic. Requiring it to vanish identically gives four independent linear conditions on $V$. Therefore
\[
V_L=\{F:L\subseteq S_F\}
\]
has rank $16$.

The singular forms do not fill $V_L$, since there exist smooth cubic surfaces containing $L$. Hence
\[
\#\{F\in V_L(H):S_F\text{ smooth}\}
\asymp H^{16}.
\]

Every such surface has a rational line, so
\[
G_F
\]
lies in the line stabilizer, the maximal subgroup of index $27$. Thus
\[
N_{\mathrm{exc}}(H)\gg H^{16}.
\]

This proves the lower bound in Theorem~\ref{thm:main}.

There is a useful heuristic indicating that allowing the rational line to vary should still produce total order $H^{16}$. A rational line $L$ corresponds to a point of
\[
\operatorname{Gr}(2,4).
\]
If its Pl\"ucker height is $T$, the lattice of restrictions of cubic forms to $L$ has covolume on the scale $T^6$: the exponent $6$ is the determinant exponent of the third symmetric power of a rank-$2$ lattice.

Thus one expects approximately
\[
\frac{H^{16}}{T^6}
\]
forms of height $H$ containing a given line of height $T$. Since rational lines of height at most $T$ occur on the scale $T^4$, the corresponding height sum behaves like
\[
H^{16}
\sum_L H(L)^{-6},
\]
which converges.

This is the principal heuristic behind Conjecture~\ref{conj:sharp}.

\subsection{Rational Eckardt points}

The rational-line family is not the only substantial exceptional locus.

Consider cubic forms
\[
F=
x_0Q_2(x_0,x_1,x_2,x_3)
+
B_3(x_1,x_2),
\]
where $Q_2$ is an arbitrary quaternary quadratic and $B_3$ an arbitrary binary cubic. This is a $14$-dimensional linear family.

Let
\[
P=(0:0:0:1).
\]
For a general member, the tangent plane at $P$ is
\[
x_0=0,
\]
and
\[
S_F\cap\{x_0=0\}
=
\{B_3(x_1,x_2)=0\},
\]
which is the union over $\overline{\Q}$ of three lines through $P$. Thus $P$ is a rational Eckardt point and the Galois group preserves the associated tritangent plane.

If $B_3$ is irreducible over $\Q$, the surface contains no rational line. Indeed, the only rational point of this tritangent-plane section is $P$, and any line on the surface through $P$ must lie in the tangent plane.

Since reducible binary cubics form a lower-order subset and smoothness is an open condition, this gives:

\begin{proposition}
There are
\[
\gg H^{14}
\]
smooth cubic forms of height at most $H$ whose line Galois group is proper but which have no rational line.
\end{proposition}

Thus even if Conjecture~\ref{conj:sharp} is correct, a secondary term of geometrically different origin should be expected.

\section{The index-$2$ cover as an arithmetic variety}

The quadratic cover
\[
\mathcal Q:
\qquad
y^2=\Delta(F)
\]
deserves independent study.

Naively, since
\[
\deg\Delta=32,
\]
one might expect square values to be extremely rare. This reasoning is unreliable: the cover contains arithmetic subvarieties on which the discriminant acquires unexpected factorization or square structure.

Elsenhans and Jahnel observed precisely this phenomenon in the pentahedral family. In pentahedral coordinates they construct a generalized Cremona transformation which reduces the apparent degree of the discriminantal covering from $32$ to $8$, and they exhibit accumulating subvarieties carrying rational points \cite{ElsenhansJahnel2012}.

This gives a warning against treating $\Delta(F)$ as a random integer of size $H^{32}$. The central question is instead geometric:

\begin{problem}
Classify the accumulating subvarieties of
\[
y^2=\Delta(F)
\]
with respect to coefficient height.
\end{problem}

A satisfactory solution of this problem might give a sharp estimate for $N_+(H)$, rather than merely the exponent $19+1/21$.

For comparison, in the reciprocal-polynomial setting the analogous quadratic obstruction is sufficiently explicit to be counted asymptotically and turns out to be the dominant exceptional family \cite{ABO2026}. In the present problem, it is not yet clear whether the index-$2$ family can compete with the $H^{16}$ rational-line contribution.

\section{Conjectures and further questions}

We conclude with three increasingly strong targets.

\begin{conjecture}[van der Waerden bound]
\[
N_{\mathrm{exc}}(H)\ll H^{19}.
\]
\end{conjecture}

By Proposition~\ref{prop:nonquadratic}, this is entirely an index-$2$ problem.

Even the slightly weaker estimate
\[
N_+(H)\ll_\varepsilon H^{19+\varepsilon}
\]
would remove the $1/21$ loss in Theorem~\ref{thm:main}.

\begin{conjecture}[sharp exponent]
\[
N_{\mathrm{exc}}(H)=H^{16+o(1)}.
\]
\end{conjecture}

This predicts that no exceptional Galois-group family has a larger exponent than the locus of cubic surfaces with a rational line.

\begin{conjecture}[rational lines dominate]
There exists $c>0$ such that
\[
N_{\mathrm{exc}}(H)\sim cH^{16},
\]
and the main term is accounted for by cubic surfaces containing a rational line.
\end{conjecture}

The third conjecture is deliberately more speculative. In particular, it requires a precise understanding of the quadratic discriminant cover and of all possible accumulating subvarieties.

Several further questions suggest themselves.

First, what is the true order of
\[
N_+(H)?
\]
The general square sieve gives
\[
H^{19+1/21+\varepsilon},
\]
but there is at present a large gap between this and the explicitly known accumulating families.

Second, can the weakly square-divisible range in Section~5 be interpreted geometrically or arithmetically in a way analogous to the passage from polynomial discriminant to number-field discriminant in Bhargava's proof?

Third, can one prove directly that surfaces with rational lines contribute
\[
c_{\mathrm{line}}H^{16}+o(H^{16})
\]
after summing uniformly over rational lines?

Finally, it would be natural to formulate analogous problems for other exceptional Weyl groups arising as monodromy groups of classical enumerative problems. From this perspective, the classical $A$-type problem of van der Waerden, the $B$-type reciprocal-polynomial problem, and the present $E_6$-problem appear as the first instances of a broader arithmetic statistics of Weyl-group monodromy.

\end{document}